\documentclass[a4paper,draft]{amsart}
 \usepackage{amsmath,amssymb,amsthm,latexsym}
 \theoremstyle{definition}
 \newtheorem{ddd}{Definition}[section]
 \theoremstyle{plain}
 \newtheorem{ttt}[ddd]{Theorem}
 \newtheorem{llll}[ddd]{Lemma}
 \newtheorem{ccc}[ddd]{Corollary}
 \newcommand{\mdeg}{\mathrm{mdeg}}
 
\begin{document}
  \title[On Universal Equivalence]{On Universal Equivalence of Partially Commutative Metabelian Lie Algebras.\footnote{The author is
        supported by the Russian Foundation for Basic Research (Grant
       12-01-00084).}}
  \author{E.\,N.\,Poroshenko}
  \date{}
  \begin{abstract}
    In this paper, we consider partially commutative metabelian Lie algebras whose defining graphs are
    cycles. We show that such algebras are universally equivalent iff the corresponding cycles have the same length.
    Moreover, we give an example showing that the class of
    partially commutative metabelian Lie algebras such that their
    defining graphs are trees is not separable by universal theory in the class of
    all partially commutative metabelian Lie algebras.
  \end{abstract}
  \maketitle
  \section{Introduction}
    Let
    $G=\langle X,E\rangle$ be an undirected graph without loops with the finite  set of vertices
    $X=\{x_1,\dots x_n\}$ and the set of edges
    $E$
    ($E\subseteq X\times X$). We denote the elements of
    $E$ by
    $\{x,y\}$.

    Suppose that
    $R$ is an associative
    commutative ring
    $R$ with a unit.
    A \emph{partially commutative Lie algebra} over
    $R$ is an
    $R$-algebra
    $\mathcal{L}_R(X;G)$ with the set of generators
    $X$ and the set of defining relations of the form
    $$[x_i,x_j]=0, \text{ where } \{x_i,x_j\}\in E.$$
    Henceforth, Lie product of
    $g$ and
    $h$ is denoted by
    $[g,h]$.
    $G$ is called the \emph{defining graph} of the corresponding algebra.

    One can also define partially commutative Lie algebras in some variety of Lie algebras.
    In this case, a partially commutative Lie algebra in a variety is a Lie
    $R$-algebra defined by a set of generators, defining
    relations, and the set of identities holding in this variety. In this paper, we consider partially
    commutative Lie algebras in the variety of metabelian Lie algebras.

    So, the definition of partially commutative Lie algebras is analogous to ones of other
    partially commutative structures such as groups, monoids etc. (see~%
 \cite{DK93}).

    Partially commutative groups  are studied very heavily nowadays (see~%
\cite{Se89,DK92,She05,She06,DKR07,GT09,Ti10,Ti11,GT11}).
    In some papers (for example, in
\cite{GT09,Ti10,GT11}), universal theories of partially
    commutative metabelian groups were studied.

    Partially commutative algebras (both associative and Lie algebras) were studied less. Although, there are
    some results obtained for other partially commutative structures (see
\cite{KMNR80,CF69,DK92',Por11,Por12,PT13}). In
 \cite{PT13}, for instance, partially commutative Lie algebras whose defining graphs are trees
    were considered and the universal theories of these algebras
    were studied.

    In this paper, we continue studying universal theories of
    partially commutative Lie algebras.

    In Section~%
\ref{sec2}, preliminary definitions and results are given.

   In Section~%
\ref{sec3}, the partially commutative Lie algebras
   whose graphs are cycles are considered. It is shown that two such algebras are universally equivalent iff
   the cycles have the same length.

   In Section~%
\ref{sec4}, some graph transformation is defined. It is shown
   that if a partially commutative metabelian Lie algebra is obtained from another one by applying this transformation
   then both algebras have the same  universal theory. The paper finishes by giving an example of
   two universally equivalent algebras such that the defining graph of the first one is a tree while the defining graph
   of the second one is not.

 \section{Preliminaries}\label{sec2}
   Let
   $R$ be an integral domain and let
   $G=\langle X;E \rangle$ be an undirected graph with the set of vertices
   $X$ and the set of edges
   $E$. By
   $M(X;G)$  denote the partially commutative
   metabelian Lie algebra with the set of generators
   $X$ and the defining graph
   $G$, i.e. the Lie
   $R$-algebra
   $M_R(X)/I$, where
   $M_R(X)$ is the free metabelian Lie
   $R$-algebra with the set of generators
   $X$ and
   $I$ is the ideal of
   $M_R(X)$ generated by the set of relations
   $\{[x_i,x_j]=0\,|\,x_i,x_j\in X, \text{ such that } \{x_i,x_j\}\in E \}$.

   If
   $\{x,y\}\in E$ then we write
   $x\leftrightarrow y$. Similarly, suppose that
   $X' \subseteq X$ and
   $x\in X$ is adjacent to all vertices in
   $X'$. Then we write
   $x \leftrightarrow X'$. These pieces of notation are going to be used as global ones. Namely, they mean that
   a vertex is adjacent to another vertex (or to all vertices in a subset of
   $X$) in
   $G$, i.e in the defining graph of the {\bfseries initial} partially commutative metabelian Lie
   algebra
   $M(X;G)$.

   In particular, since
   $G$ has no loops
   $x\leftrightarrow X'$ implies
   $x\not \in X'$.

   In this paper, we assume
   $X=\{x_0,x_1, \dots, x_{n-1}\}$. We also suppose that
   $R$ is an integral domain containing the ring of integers
   $\mathbb{Z}$ as a subring. We denote Lie monomials as bracketed lowercase Latin letters (for
   example,
   $[u]$) to keep the same notation as in some other papers.
   \begin{ddd} Let
     $[u]$ be a Lie monomial in the set of generators
     $X$. The \emph{multidegree} of
     $[u]$ is the vector
     $\overline{\delta}=(\delta_{0},\delta_2,\dots,\delta_{n-1})$, where
     $\delta_i$ is the number of occurrences of
     $x_i$ in
     $[u]$.
   \end{ddd}

   Let us denote by
   $\mdeg ([u])$ the multidegree of
   $[u]$ and by
   $\mdeg_i (u)$ the number of
   $x_i$ in
   $[u]$, i.e. the
   $i$th coordinate of
   $\mdeg([u])$.

   The \emph{glued multidegree} of
   $[u]$ is the vector
   $\widetilde{\mdeg}([u])=(\delta_{0},\delta_2,\dots,\delta_{n-2}+\delta_{n-1})$,
   where
   $\delta_i=\mdeg_i([u])$ as above.

   Let us introduce some pieces of notation associated with
   graphs. Let
   $H$ be an arbitrary undirected graph. By
   $V(H)$ and
   $E(H)$ denote the set of the vertices and the set of the edges of this graph respectively. Next, let
   $V'\subseteq V(H)$. By
   $H(V')$ denote the subgraph of
   $H$, generated by the set
   $V'$.

   Let
   $[u]$ be a Lie monomial not equal to zero in
   $M(X;G)$.  By
   $X_{[u]}$ denote the set
   $\{x_i\in X\,|\, \mdeg_{i}([u])\neq 0\}$. Correspondingly,
   $G_{[u]}$ is a subgraph of
   $G$ generated by
   $X_{[u]}$. Likewise, if
   $v$ is an associative monomial then denote by
   $X_{v}$ the set of generators occurring in
   $v$. We also denote by
   $G_{v}$ the graph
   $G(X_v)$.

   In
\cite{PT13}, the explicit description of a basis of a Lie
   algebra
   $M(X;G)$ was found. Suppose that the set
   $X$ is linearly ordered. Then the basis of
   $M(X;G)$ is the set of monomials of the form
   $[u]= [\dots [x_{i_1},x_{i_2}],\dots, x_{i_k}]$ having the following properties:
   \begin{enumerate}
     \item
       $x_{i_2}<x_{i_1}$;
     \item
       $x_{i_2}\leqslant x_{i_3}\leqslant \dots \leqslant x_{i_k}$;
     \item
       $x_{i_1}$ and
       $x_{i_2}$ are in different connected components of
       $G_{[u]}$;
     \item
       let
       $H$ be the connected component of
       $G_{[u]}$ containing
       $x_{i_1}$, then
       $x_{i_1}$ is the greatest element among the vertices in
       $V(H)$.
   \end{enumerate}

   Such monomials are called \emph{basis monomials}. Note, that the set of
   basis monomials depends on an order of the set
   $X$. However, whatever ordering is used the set of all basis
   monomials of
   $M(X;G)$ forms a basis of this algebra.

   \begin{ddd}
     If all monomials of a Lie polynomial
     $g$ have the same multidegree
     $\overline{\delta}$ then we call such polynomial
     \emph{homogeneous} and write
     $\mdeg (g)=\overline{\delta}$, where
     $\overline{\delta}$ is the multidegree of a monomial in
     $g$.
   \end{ddd}

   Since the set of identities and the set of defining relations
   of a partially commutative metabelian Lie algebra are
   homogeneous the following statement holds.
   If
   $0=\sum_{i}g_i$ in
   $M(X;G)$, where all
   $g_i$ are homogeneous Lie polynomials of pairwise distinct multidegrees, then
   for any
   $i$ we have
   $g_i=0$ in this algebra. In particular, if
   $g$ is homogeneous then for any order of
   $X$ all basis monomials in
   $g$ have the same multidegree.

   Let
   $g$ be an element such that its decomposition
   to the linear combination of basis monomials is a homogeneous
   polynomial. Then for any other order of
   $X$ the decomposition of
   $g$ to the linear combination of basis monomials is also
   homogeneous. Moreover, the multidegrees of the corresponding
   polynomials are same for all orders. So we can define  \emph{homogeneous elements}
   in an obvious way. Indeed, since all identities and defining relations of a
   partially commutative metabelian Lie
   $R$-algebra are homogeneous, any transformation of a non-zero Lie
   monomial gives a homogeneous linear combination of the same
   multidegree. So, we can also define the \emph{multidegree} of a homogeneous
   element of
   $M(X;G)$ as follows. Let
   $g$ be a homogeneous element. Consider a decomposition of this element to a linear combination of basis monomials
   (with respect to any
   order on
   $X$). Then by definition put
   $\mdeg(g)=\mdeg([u])$, where
   $[u]$ is an arbitrary basis Lie monomial contained in this decomposition.

   Let
   $[u]=[\dots[x_{i_1},x_{i_2}],\dots, x_{i_k}]$. It is easy to show (see, for example,
\cite{PT13}) that for any permutation of
   $x_{i_3},\dots, x_{i_k}$ we obtain a monomial equal to
   $[u]$ in
   $M'(X)$, consequently these monomials are also equal in
   $M'(X;G)$. That means
   \begin{equation}\label{permut}
      [\dots[[x_{i_1},x_{i_2}],x_{i_3}],\dots, x_{i_k}]=
        [\dots[[x_{i_1},x_{i_2}],x_{\sigma(i_3)}],\dots, x_{\sigma(i_k)}],
   \end{equation}
   where
   $\sigma$ is a permutation of
   $\{i_3,i_4,\dots , i_k\}$.

   Let
   $R[X]$ be the set of all commutative associative polynomials over
   $R$. It follows from the last paragraph that the derived subalgebra
   $M'(X;G)$ of
   $M(X;G)$ is an
   $R[X]$-module with respect to the adjoint representation. Denote by
   $u.f$ the element of
   $M'(X;G)$ obtained by acting the element
   $f\in R[X]$ on
   $[u]\in M'(X;G)$. Namely, let us define
   $u.f$ inductively:
   \begin{enumerate}
     \item
       $u.y=[u,y]$ for any
       $y\in X$;
     \item
       Let
       $f=y_1y_2\dots y_m$ for
       $m\geqslant 2$ and let
       $f_0=y_1y_2\dots y_{m-1}$ then
       $u.f=(u.f_0).y_m$;
     \item
       Finally, if
       $f=g+s$, where
       $s$ is a commutative associative monomial then
       $u.f=u.g+u.s$.
   \end{enumerate}

   \begin{ddd}
     Let
     $g$ be an arbitrary element of the algebra
     $M(X;G)$  and let
     $C(g)$ be the centralizer of
     $g$. The set
     $$\mathcal{C}(g)=C(g)\cap M'(X;G)$$
     is called the \emph{derived centralizer} of
     $g$.
   \end{ddd}

   For the derived centralizers of the generators of
   $M(X;G)$ the following theorem holds (see~%
\cite{PT13}).
   \begin{ttt}\label{centinter}
     Let
     $M(X;G)$ be a metabelian partially commutative Lie
     $R$-algebra, where
     $X=\{x_0,x_1,\dots,x_{n-1}\}$. Then for any elements
     $x_{i_1},x_{i_2},\dots, x_{i_m}$ and for any
     $\alpha_{i_1},\alpha_{i_2},\dots,\alpha_{i_m}\in R\backslash\{0\}$
     we have
     $$\mathcal{C}\bigl(\sum_{j=1}^m \alpha_{i_j} x_{i_j}\bigr)=\bigcap_{j=1}^{m} \mathcal{C}(x_{i_j})$$
   \end{ttt}

   Finally, let us remind some terminology related universal theories of algebraic systems.
   Let
   $\Phi$ be a formula having no free variables and including elements of an algebra
   $A$. Then by definition put
   $A \models \Phi$ if
   $\Phi$ holds on
   $A$.
   \begin{ddd}
     An
     $\exists$-\emph{sentence} is a formula of the form
     $$\exists w_1\dots w_m \Phi(w_1,\dots,w_m),$$
     having no free variables. Here
     $\Phi(w_1,\dots,w_m)$ is a formula of predicate calculus in
     the corresponding algebraic system such that this formula does not contain quantifiers.
   \end{ddd}

   \begin{ddd}
     The set of all
     $\exists$-sentences that are true in a Lie algebra
     $L$ is called the \emph{existential theory} or the
     $\exists$-\emph{theory} of this Lie algebra.
   \end{ddd}
   \begin{ddd}
     Lie algebras are called \emph{existentially equivalent}
     if their existential theories coincide.
   \end{ddd}
   The notion of
   $\forall$-\emph{sentence} is defined analogously as well as
   the notions of \emph{universal theory}(or
   $\forall$-\emph{theory}) of a Lie algebra and \emph{universally equivalent} Lie algebras.

   It is easy to see that Lie algebras
   $L_1$ and
   $L_2$ are existentially equivalent iff these Lie algebras are
   universally equivalent.

   The procedure of exchanging functional symbols by predicate ones is well known in model theory.
   Any set with all predicates induced on it is a submodel.

  Let us formulate a well-known result in model theory.
  \begin{ttt}\label{univeq}
    Arbitrary algebraic systems (ex., Lie algebras)
    $L_1$ and
    $L_2$ are universally equivalent iff each finite model of the first algebraic system
    is isomorphic to a finite model of the second one.
  \end{ttt}

 \section{Algebras whose defining graphs are cycles}\label{sec3}
   In this section, we consider only the algebras
   $M(X;C_n)$, where
   $C_n$~--- is a cycle on
   $n$ vertices
   ($n\geqslant 3$).

   By
   $\mathbb{Z}_n$ denote the residue ring with respect to base
   $n$, namely put
   $\mathbb{Z}_n=\{0,1,2,\dots, n-1\}$ with addition and subtraction defined in the natural way (for residue rings).
   For arbitrary
   $r,s\in \mathbb{Z}_n$ let us denote by
   $|r-s|$ the lesser element between
   $r-s$ and
   $s-r$ (in the sense of the standard order:
   $0<1<\dots< n-1$).

   Let
   $G$ be a graph. A connected component of
   $G$ is called a \emph{chain} if it is an isolated vertex or a tree with two endpoints. In other words,
   a chain is a tree in that all vertices are connected consecutively.

   \begin{llll} \label{centralizators}
     Let
     $M(X;C_n)$ be the metabelian partially commutative Lie
     $R$-algebra, whose defining graph is the cycle of the length
     $n$, where
     $n\geqslant 3$. Then the following properties hold.\\
     a) If
     $\alpha,\beta \in R\backslash\{0\}$ then
     $\mathcal{C}(\alpha x_i+\beta x_{i+1})=0$.\\
     b) Let
     $|i-j|>1$ and
     $\alpha,\beta \in R\backslash\{0\}$. Then the derived centralizer
     $\mathcal{C}(\alpha x_i+\beta x_j)$ consists of all linear combinations of non-zero Lie monomials
     $[u_r]$ such that
     $X_{[u_r]}=X\backslash\{x_i,x_j\}$. Moreover, any element of
     $\mathcal{C}(\alpha x_i+\beta x_j)$ can be represented in the form
     $f=[x_{i-1},x_{i+1}].g$, where
     $g$ is an associative polynomial.\\
     c) If
     $m\geqslant 3$ and
     $\alpha_1,\dots, \alpha_m \in R\backslash \{0\}$ then
     $\mathcal{C}(\sum_{j=1}^m \alpha_j x_{i_j})=0$.
   \end{llll}
   \begin{proof}
     It follows from Theorem~%
\ref{centinter} that for any
     $\alpha_{j}\in R\backslash\{0\}$ ($j=1,2,\dots, m$) the equation
     $\mathcal{C}(\sum_{j=1}^m \alpha_j x_{i_j})=\bigcap_{j=1}^m \mathcal{C}(x_{i_j})$ holds in any algebra
     $M(X;G)$.

     By construction of basis monomials we can see that if
     $[u] \in \mathcal{C}(x_i)\backslash \{0\}$ for a Lie monomial
     $[u]$ then
     $\mdeg_i([u])=0$.

     a) Let
     $g\in \mathcal{C}(\alpha x_i+\beta x_{i+1})$ and let
     $g\neq 0$ in
     $M(C_n;X)$. By
\cite{PT13}, we only have to prove the assertion for non-zero
     Lie monomials. In this case,
     $G_{[u]}$ has at least two chains. Therefore, there exists
     $j\in \mathbb{Z}_n$ such that
     $x_j \not \in X_{[u]}$ and there is a representation
     $[u]=[x_s,x_t].w$, where
     $w$ is an associative monomial and
     $x_s, x_t$ are generators from different chains. Besides, since
     $[u].x_i=[u].x_{i+1}=0$ these generators should be in the same chain of the graphs
     $G(X_{[u]}\backslash \{x_i\})$ and
     $G(X_{[u]}\backslash \{x_{i+1}\})$ but this is impossible.
     Indeed,
     $x_i$ and
     $x_{i+1}$ are adjacent in
     $C_n$. Therefore, we have to add both vertices to
     $G_{[u]}$ to get one chain in the obtained graph instead of two ones in
     $G_{[u]}$.

     b) Let
     $g\in \mathcal{C}(\alpha x_i+\beta x_{j})$, where
     $|i-j|>1$. As in a), let us suppose that
     $g\neq 0$ in
     $M(X;C_n)$. Again it is sufficient to prove the statement for non-zero Lie monomials.
     Let
     $[u]$ be a monomial satisfying the conditions of b). We can
     write
     $[u]=[x_t,x_s].w$ for some associative monomial
     $w$. Let
     $T_1,T_2,\dots, T_p$ be the chains of
     $G(X_{[u]})$
     ($p\geqslant 2$ because
     $[u]\neq 0$ in
     $M(X;C_n)$). We can assume without loss of generality that
     $x_s\in T_1$, and
     $x_t\in T_2$. Since
     $[u] \in \mathcal{C}(x_i)$ we have
     $[u].x_i=0$. Consequently the vertices of the chains
     $T_1$ and
     $T_2$ are in the same connected component of
     $G(X_{[u]}\backslash \{x_i\})$. So,
     $x_i$ connects the subgraph
     $T_1$ with the subgraph
     $T_2$. Namely, in the graph
     $C_n$, the vertex
     $x_i$ is adjacent to some vertex of
     $T_2$ the vertex
     $x_i$ is adjacent to some vertex of
     $T_2$ as well. Using the same arguments for
     $x_j$ we obtain that in
     $C_n$, the vertex
     $x_j$ is adjacent to a vertex in
     $T_1$ and
     $x_j$ is adjacent to a vertex in
     $T_2$ as well. This is possible only if there are only two chains in
     $G_{[u]}$ and one endpoint of each chain is adjacent to
     $x_i$  while the other endpoint of each chain is adjacent to и
     $x_j$. Consequently,
     $X_{[u]}=X \backslash \{x_i,x_j\}$.

     There are only two connected components in
     $G_{[u]}$ and the vertices
     $x_{i-1}$ and
     $x_{i+1}$ are in the different connected components. Therefore, one of these vertices is in the same
     connected component with
     $x_s$ and the other one is in the same connected component with
     $x_t$.

     In
\cite{PT13}, it was shown that if
     $[w_1]$ is a non-zero Lie monomial and
     $[w_2]$ is obtained from
     $[w_1]$ by switching two generators
     from the same connected
     component of
     $G_{[w_1]}$, then
     $[w_1]=[w_2]$ in
     $M(X;G)$.

     Thus, we can switch
     $x_{i-1}$ with the vertices in
     $\{x_{s},x_{t}\}$ lying in the same connected component as
     $x_{i-1}$ and
     $x_{i+1}$ with another vertex of this set. Applying the anticommutativity identity to the obtained
     monomial (if it is necessary) we obtain a monomial in the form
     $[x_{i+1},x_{i-1}].v$, which is equal to
     $[u]$ in
     $M(X;G)$.

     c) It is sufficient to show that
     $\mathcal{C}(\alpha x_{i}+\beta x_j+\gamma x_{k})=0$.
     If at least two among these three vertices are adjacent in
     $C_n$ then the proof follows from  a).

     Suppose that
     $|i-j|>1$,
     $|j-k|>1$ and
     $|i-k|>1$. Let
     $T_1,T_2,\dots T_p$ are chains of
     $G_{[u]}$. Obviously,
     $p\geqslant 3$ in this case. All arguments are similar to ones in b). Let
     $[u]=[x_r,x_s].w$, where
     $w$ is an associative monomial. We can assume without loss of generality that
     $x_r$ is in a chain
     $T_1$ and
     $x_s$ is in a chain
     $T_2$. Since
     $[u].x_i=0$, we get
     $x_i$ is adjacent to some vertex in
     $T_1$ and
     $x_i$ is adjacent to some vertex in
     $T_2$ as well. That means that
     $x_i$ connects
     $T_1$ and
     $T_2$ in a connected component in
     $G(X_{[u]}\backslash \{x_i\})$. Analogously, since
     $[u].x_j=0$, we obtain
     $x_j$ is is adjacent to some vertex in
     $T_1$ and
     $x_1$ is adjacent to some vertex in
     $T_2$ as well. But in this case,
     $T_1$ and
     $T_2$ together should contain all vertices  in
     $X\backslash \{x_i,x_j\}$. This  is impossible because there are at least
     tree chains. So the proof is completed.
   \end{proof}

   Let
   $m\geqslant 4$.
   Consider the formula.
   \begin{equation}\label{longformula}
      \Phi(m)= \exists z_0,z_1,  \dots, z_{m-1} \Theta(z_0,z_1,\dots,z_{m-1}),
   \end{equation}
   where
   \begin{equation*}
     \begin{split}
       \Theta(z_0,z_1,\dots,z_{m-1})= &
         \left(\bigwedge_{i\in \mathbb{Z}_m} [z_i,z_{i+1}]=0
         \wedge \hspace{-2mm}\bigwedge_{i,j\in \mathbb{Z}_m: |j-i|>1}\hspace{-4mm} [z_i,z_j]\neq 0 \right. \wedge\\
         & \wedge \hspace{-2mm}\left.\bigwedge_{i,j \in \mathbb{Z}_m: |i-j|\cdot |(i+2)-j|\neq 1}\hspace{-4mm}
         [[z_{i},z_{i+2}],z_j]\neq 0  \right).
     \end{split}
   \end{equation*}
   It is easy to see that this formula holds in
   $M(X,C_m)$. Indeed, let us put
   $g_i$ to be equal
   $x_i$ for any
   $i\in \mathbb{Z}_m$. Obviously,
   $[x_i,x_j]=0$ iff
   $|i-j|\leqslant 1$. We also can see that
   $[[x_{i},x_{i+2}],x_{j}]=0$ iff
   the vertices
   $x_{i}$ and
   $x_{i+2}$ are in the same connected component of
   $G(\{x_i,x_{i+2},x_j\})$. But this is true only in the case
   \begin{equation} \label{prop}
     |j-i|=|j-(i+2)|=1.
   \end{equation}

   Note that if
   $m\geqslant 5$ then for any
   $k,l\in \mathbb{Z}_m$ such that
   $|k-l|=2$ there is a unique
   $t$ such that
   $|k-t|=|l-t|=1$. So, if
   $k-t=1$ then we may assume that
   $l=i$, then
   $k=i+2$ and
   $t=i+1=j$. Otherwise, we suppose that
   $k=i$, then
   $l=i+2$ and again
   $t=i+1=j$.

   However, if
   $m=4$ then the assertion from the last paragraph is not true. Indeed,
   let
   $k,l \in \mathbb{Z}_4$ be such that
   $|k-l|=2$. Then for either
   $t\in \mathbb{Z}_4$ such that
   $t\neq k$ and
   $t\neq l$ we have
   $|k-t|=|l-t|=1$. As in the last paragraph, we can put
   $k=i$ if
   $l-t=1$ and
   $l=i$ in the other case.

   Let us prove some conditions the elements
   $g_0,g_1,\dots, g_{m-1}\in M(X;C_n)$ should satisfy to
   $M(X;C_n)\models \Theta(g_0,g_1,\dots,g_{m-1})$ true. But before doing that, let us note that any
   $g_i$ can be written as follows:
   \begin{equation*}
     g_i=\widehat{g}_i+g'_i,
   \end{equation*}
   where
   $g'_i\in M'(X;C_n)$ and
   \begin{equation} \label{linearpart}
     \widehat{g}_i=\sum_{l=1}^{r_i}\alpha_{i,l} x_{k_{i,l}},
   \end{equation}
   moreover, for any fixed
   $i$ all indices
   $k_{i,l}$ are distinct and
   $\alpha_{i_k}\in R\backslash\{0\}$. So,
   $\widehat{g}_i$ is the linear part of
   $g_i$. Thus, we have
   \begin{equation} \label{2prod}
     \begin{split}
       [g_i,g_j] &=[\widehat{g}_i+g'_i,\widehat{g}_j+g'_j]=\\
                 &=[\widehat{g}_i,\widehat{g}_j]+[\widehat{g}_i,g'_j]+
                   [g'_i,\widehat{g}_j]+[g'_i,g'_j]= \\
                 &=[\widehat{g}_i,\widehat{g}_j]-[g'_j,\widehat{g}_i]+
                   [g'_i,\widehat{g}_j]
     \end{split}
   \end{equation}
   Besides, we need the following formula:
   \begin{equation}\label{3prod}
     \begin{split}
       [[g_{i},g_{j}],g_k]=& [[g_{i},g_{j}],\widehat{g}_k+g'_k]= \\
         =& [[g_{i},g_{j}],\widehat{g}_k]+[[g_{i},g_{j}],g'_k]= \\
         =& [[g_{i},g_{j}],\widehat{g}_k].
     \end{split}
   \end{equation}

   \begin{llll}\label{linpartexists}
     Let
     $m\geqslant 5$ and let
     $g_0,g_1,\dots,g_{m-1}$ be the elements of
     $M(X; C_n)$
     ($n\geqslant 3$) such that
     $M(X;C_n)\models \Theta(g_0,g_1,\dots,g_{m-1})$. Then for all
     $i\in \mathbb{Z}_m$ we have
     $\widehat{g}_{i}\neq 0$ in
     $M(X;C_n)$.
   \end{llll}
   \begin{proof}
     Let
     $\widehat{g}_i=0$ for some
     $i$. Since
     $m\geqslant 5$ we have
     $|(i+3)-i|>1$. Obviously,
     $[[g_{i+3},g_{i+1}],g_i]=0$ because
     $M(X,C_n)$ is metabelian and
     $[g_{i+3},g_{i+1}],g_i \in M'(X;C_n)$.We get a contradiction to
(\ref{longformula}). Therefore, for any
     $i\in \mathbb{Z}_m$ we obtain
     $\widehat{g}_i \neq 0$.
   \end{proof}

   \begin{llll} \label{no3}
     Let
     $m\geqslant 5$ and let
     $g_0,g_1,\dots,g_{m-1}$ be the elements of
     $M(X; C_n)$
     ($n\geqslant 4$) such that
     $M(X;C_n)\models \Theta(g_0,g_1,\dots,g_{m-1})$. Then for all
     $i\in \mathbb{Z}_m$ we have
     $r_i\leqslant 2$, where
     $r_i$ are taken in
(\ref{linearpart}). In other words, the linear part of each
     $g_i$ has at most two summands. Moreover, if
     $r_i=2$ for some
     $i\in \mathbb{Z}_m$ then
     $|k_{i,1}-k_{i,2}|>1$, where
     $k_{i,1}$ and
     $k_{i,2}$ are also taken in
(\ref{linearpart})).
   \end{llll}
   \begin{proof}
     Let
     $\widehat{g}_i$ have at least three non-zero summands for some
     $i$, i.e. let
     $r_i\geqslant 3$.
     Since
     $[g_{i+1},g_{i-1}],g'_i\in M'(X;C_n)$,
(\ref{3prod}) implies
     \begin{equation*}
       0= [[g_{i+1},g_{i-1}],\widehat{g}_i]=
         \left[[g_{i+1},g_{i-1}],\sum_{l=1}^{r_i}\alpha_{i,l} x_{k_{i,l}}\right].
     \end{equation*}
     Hence,
     $[g_{i+1},g_{i-1}] \in \mathcal{C}\left(\sum_{l=1}^{r_{i}} \alpha_{i,l} x_{k_{i,l}}\right)$.
     By Lemma~%
\ref{centralizators}~c),
     $[g_{i+1},g_{i-1}]=0$. This contradicts to
(\ref{longformula}) because for
     $m\geqslant 5$ we have
     $|(i+1)-(i-1)|>1$.

     Let
     $\widehat{g}_i= \alpha_{i,1} x_{j}+\alpha_{i,2} x_{j+1}$ for some
     $j\in \mathbb{Z}_m$ and for
     $\alpha_{i,1},\alpha_{i,2} \in R\backslash\{0\}$ (i.e in the case of
     $r_i=2$ and
     $|k_{i,1}-k_{i,2}|=1$). Then by Lemma~%
\ref{centralizators}~a),
     $[g_{i+1},g_{i-1}]=0$ and
     we also obtain a contradiction.
   \end{proof}

   \begin{llll}\label{noswitchr}
     Let
     $m\geqslant 5$ and let
     $g_0,g_1,\dots,g_{m-1}$ be the elements of
     $M(X; C_n)$
     ($n\geqslant 4$) such that
     $M(X;C_n)\models \Theta(g_0,g_1,\dots,g_{m-1})$. If
     $r_{i-1}=r_{i+1}=1$ for some
     $i\in \mathbb{Z}_m$, then
     $r_i=1$ where
     $r_{i-1},r_i,r_{i+1}$ are taken in
(\ref{linearpart}).
   \end{llll}
   \begin{proof}
     Suppose that $r_{i-1}=r_{i+1}=1$ and
     $r_i=2$ for some
     $i\in \mathbb{Z}_m$. So, let
     $\widehat{g}_i=\alpha x_j+\beta x_k$ (where
     $|j-k|>1$),
     $\widehat{g}_{i-1}=\gamma x_s$, and
     $\widehat{g}_{i+1}=\delta x_t$.

     By
(\ref{longformula}), we have
     $[g_{i-1},g_i]=0$. Therefore, by
(\ref{2prod}) and homogeneity of identities and relations in a
     metabelian partially commutative Lie algebra   in
     $M(X;C_n)$, we obtain
     $[\widehat{g}_{i-1},\widehat{g}_{i}]=0$. That is
     $[\alpha x_j+\beta x_k,\gamma x_s]=\alpha\gamma  [x_j,x_s]+\beta\gamma[x_k,x_s]=0$.
     Consequently,
     $|j-s|\leqslant 1$ and
     $|k-s|\leqslant 1$. There are two cases possible: either
     $s$ is equal to one of the elements
     $j,k$ or it is distinct from both of them.\\

     \noindent
     1. We can assume without loss of generality that
     $j=s$. Since
     $r_i=2$, we have
     $j\neq k$. Therefore,
     $|k-j|=1$. But this contradicts to Lemma~%
\ref{no3}.\\

     \noindent
     2. Let
     $|j-s|=|k-s|=1$. We can assume without loss of generality that
     $j=s-1$,
     $k=s+1$.

     Consider
     $[g_{i},g_{i+1}]$. By the same argument as above,
     $|j-t|=|k-t|=1$. We again get two cases.\\

     \noindent
     2-1. Let
     $t=s$. By
(\ref{longformula}),
     $[[g_{i-2},g_i],g_{i-1}]=0$. Consequently, by~%
(\ref{3prod}), we get
     $[[g_{i-2},g_i],\widehat{g}_{i-1}]=0$. That means
     $[g_{i-2},g_i] \in \mathcal{C}(x_s)$. Now, by~%
(\ref{3prod}), we obtain
     $[[g_{i-2},g_i],g_{i+1}]=[[g_{i-2},g_i],\widehat{g}_{i+1}]=0$
     (remind that
     $\widehat{g}_{i+1}=\delta x_s$). We get a contradiction to
(\ref{longformula}) because
     $|(i-2)-(i+1)|\neq 1$ in
     $\mathbb{Z}_m$ for
     $m\geqslant 5$.\\

     \noindent
     2-2. Let
     $t\neq s$. This is possible only if
     $n=4$. Renumbering generators if it is necessary, we can suppose that
     $\widehat{g}_i=\alpha x_1 +\beta x_3$,
     $\widehat{g}_{i-1}=\gamma x_0$,
     $\widehat{g}_{i+1}=\delta x_2$.

     Consider
     $g_{i+2}$. By
(\ref{2prod}) and homogeneity of identities and relations in a
     metabelian partially commutative Lie algebra,
     $[\widehat{g}_{i+1},\widehat{g}_{i+2}]=0$. Remind that
     $[x_p,x_q]=0$ iff
     $|p-q|\leqslant 1$. Consequently,
     $\widehat{g}_{i+2}=\zeta x_1+\eta x_2+\theta x_3$, where
     $\zeta,\eta,\theta \in R$. Moreover, by Lemma~%
\ref{no3}, either
     $\eta=0$ or
     $\zeta=\theta=0$.\\

     \noindent
     2-2-1. If
     $\eta=0$, then
     $g_{i+2}=\zeta x_1+ \theta x_3$. By
(\ref{3prod}), we have
     $[[g_{i-1},g_{i+1}],\widehat{g}_{i}]=0$. Hence,
     $[g_{i-1},g_{i+1}]\in \mathcal{C} (\widehat{g}_{i})=\mathcal{C}(x_1)\cap \mathcal{C}(x_3)$.
     But in this case
     $[g_{i-1},g_{i+1}]\in \mathcal{C}(\widehat{g}_{i+2})$ and it
     does not matter whether neither
     $\zeta$ nor
     $\theta$ is equal to
     $0$. So,
     $[[g_{i-1},g_{i+1}],\widehat{g}_{i+2}]=0$. Since
     $m\geqslant 5$, we have
     $|(i-1)-(i+2)|\neq 1$ in
     $\mathbb{Z}_m$ and we get a contradiction to
(\ref{longformula}).\\

     \noindent
     2-2-2. If
     $\zeta=\theta=0$, then Lemma~%
\ref{linpartexists} implies
     $\eta\neq 0$. We have
     $g_{i+2}=\eta x_2$. By
(\ref{3prod}), we obtain
     $[[g_{i},g_{i+2}],\widehat{g}_{i+1}]=0$. Therefore,
     $[g_{i},g_{i+2}]\in \mathcal{C} (\widehat{g}_{i+1})=\mathcal{C}(x_2)$.
     But in this case
     $[g_{i},g_{i+2}]\in \mathcal{C}(\widehat{g}_{i+2})$. So, by~%
(\ref{3prod}), we obtain
     $[[g_{i},g_{i+2}],g_{i+2}]=0$. Again we get a contradiction to
(\ref{longformula}).
   \end{proof}

   \begin{llll}\label{linpartonecomp}
     Let
     $m\geqslant 5$ and let
     $g_0,g_1,\dots,g_{m-1}$ be the elements of
     $M(X; C_n)$
     ($n\geqslant 4$) such that
     $M(X;C_n)\models \Theta(g_0,g_1,\dots,g_{m-1})$. Then
     $r_i=1$ for all
     $i\in \mathbb{Z}_m$.
   \end{llll}
   \begin{proof}
     By Lemma~%
\ref{noswitchr} we are left to show that the case of
     $r_i=r_{i+1}=2$ is impossible.

     Suppose that
     $\widehat{g}_i=\alpha x_j+ \beta x_k$ and
     $\widehat{g}_{i+1}=\gamma x_s+ \delta x_t$, where
     $\alpha,\beta,\gamma,\delta \neq 0$ and
     $|j-k|>1$,
     $|s-t|>1$. Then
(\ref{2prod}),
 (\ref{longformula}), and homogeneity of identities and relations of a partially commutative metabelian Lie
     algebra imply
     \begin{equation}\label{length2}
       0=[\widehat{g}_i,\widehat{g}_{i+1}]=\alpha\gamma[x_j,x_s]+\alpha\delta[x_j,x_t]+\beta\gamma[x_k,x_s]+\beta\delta[x_k,x_t].
     \end{equation}
     Indeed, only these summands in
(\ref{2prod}) are the products of two generators. Consider several cases.\\

     \noindent
     1. Let
     $j,k,s,t$ be distinct. Then
(\ref{length2}) and homogeneity of identities and relations of
     partially commutative metabelian Lie algebra imply
     \begin{equation} \label{2prod0}
       [x_j,x_s]=[x_j,x_t]=[x_k,x_s]=[x_k,x_t]=0.
     \end{equation}
     By the first two parts of
(\ref{2prod0}),
     $|j-s|=|j-t|=1$. Then
     $x_s,x_j,x_t$ are connected successfully in
     $C_n$, namely
     $x_j$ is adjacent to both other vertices. For the same
     reason, we obtain
     $|k-s|=|k-t|=1$ (considering two last parts of
(\ref{2prod0})). Therefore
     $x_k$ is also adjacent to both
     $x_s$, and
     $x_t$. This is possible only in the case of
     $n=4$.

     We can assume without loss of generality that
     $\widehat{g}_i=\alpha x_0+\beta x_2$, а
     $\widehat{g}_{i+1}=\gamma x_1+\delta x_3$.

     Consider
     $g_{i+2}$. Let
     $\widehat{g}_{i+2}=\zeta x_0+\eta x_1+\theta x_2+ \kappa x_3$.
     By
(\ref{3prod}),
     $[[g_{i+1},g_{i+3}],\widehat{g}_{i+2}]=0$, therefore,
     $[g_{i+1},g_{i+3}] \in \mathcal{C}(\widehat{g}_{i+2})$. It follows from Lemma~%
\ref{centralizators} that either
     $\zeta=\theta=0$ or
     $\eta=\kappa=0$.\\

     \noindent
     1-1. Let
     $\zeta=\theta=0$. Then
     $\widehat{g}_{i+2}=\eta x_1+\kappa x_3$. By
(\ref{3prod}),
     $[[g_{i},g_{i+2}],\widehat{g}_{i+1}]=0$, therefore
     $[g_{i},g_{i+2}] \in \mathcal{C}(\widehat{g}_{i+1})=\mathcal{C}(x_1)\cap \mathcal{C}(x_3)$. But in this case,
     $[g_{i},g_{i+2}] \in \mathcal{C}(\widehat{g}_{i+2})$ and it
     does not matter whether neither
     $\eta$ nor
     $\kappa$ is equal to
     $0$.
     Hence, by~%
(\ref{3prod}), we obtain
     $[[g_{i},g_{i+2}],g_{i+2}]=0$, but this contradicts to
 (\ref{longformula}).\\

     \noindent
     1-2. Let
     $\eta=\kappa=0$. Then
     $\widehat{g}_{i+2}=\zeta x_0+\theta x_2$. By analogy with
     Case\,1-1,
(\ref{3prod}) implies
     $[[g_{i-1},g_{i+1}],\widehat{g}_{i}]=0$. Therefore,
     $[g_{i-1},g_{i+1}] \in \mathcal{C}(\widehat{g}_{i})=\mathcal{C}(x_0)\cap \mathcal{C}(x_2)$. But then
     $[g_{i-1},g_{i+1}] \in \mathcal{C}(\widehat{g}_{i+2})$ and it
     does not matter whether neither
     $\zeta$ nor
     $\theta$ is equal to
     $0$. Thus, by
(\ref{3prod}), we obtain
     $[[g_{i-1},g_{i+1}],g_{i+2}]=0$. But this contradicts to
 (\ref{longformula}) because if
     $m\geqslant 5$ then
     $|(i-1)-(i+2)|\neq 1$. \\

     \noindent
     2. Let
     $\{j,k\}\cap \{s,t\}$ have one element. Without loss of
     generality, we can suppose that
     $j=s$,
     $k\neq t$. In this case,
(\ref{length2}) implies
     $\alpha\delta[x_j,x_t]+\beta\gamma[x_k,x_j]+\beta\delta[x_k,x_t]=0$.
     By homogeneity of identities and relations of a partially commutative metabelian Lie
     algebra,
     $[x_j,x_t]=[x_k,x_j]=[x_k,x_t]=0$. Since
     $j\neq k$ and
     $j\neq t$ we have
     $|j-k|=|j-t|=1$. As
     $k\neq t$, we obtain
     $|k-t|=2$. But then
     $[x_k,x_t]\neq 0$. It s a contradiction. Consequently this case is impossible.\\

     \noindent
     3. Finally, let
     $\{j,k\}=\{s,t\}$. We can assume without loss of generality that
     $s=j$ and
     $t=k$. So,
     $g_{i+1}=\gamma x_j+\delta x_k$. By
(\ref{3prod}) we obtain
     $[[g_{i-1},g_{i+1}],\widehat{g}_{i}]=0$. Thus,
     $[g_{i-1},g_{i+1}] \in \mathcal{C}(\widehat{g}_{i})=\mathcal{C}(x_j)\cap \mathcal{C}(x_k)$. Then
     $[g_{i-1},g_{i+1}] \in \mathcal{C}(\widehat{g}_{i+1})$.
     Consequently,
     $[g_{i-1},g_{i+1}]\in \mathcal{C}(\widehat{g}_{i+1})$. So, by
(\ref{3prod}) we obtain
     $[[g_{i-1},g_{i+1}],g_{i+1}]=0$, but this contradicts to
(\ref{longformula}).

     Taking into account Lemma~%
\ref{noswitchr},  we have considered all cases and have seen that
     none of them is appropriate. This completes the proof.
   \end{proof}

   \begin{ttt}
     Let
     $R$ be an integral domain having
     $\mathbb{Z}$ as a subring and let
     $X=\{x_1,x_2,\dots, x_n\}$,
     $Y=\{y_1,y_2,\dots, y_m\}$. The partially commutative metabelian Lie algebras
     $M(X;C_n)$ and
     $M(Y;C_m)$ are universally equivalent iff
     $m=n$.
   \end{ttt}
   \begin{proof}
     The converse is obvious. Indeed, if
     $m=n$ then
     $M(X;C_n)$ and
     $M(Y;C_m)$ are isomorphic.

     Let us prove the direct statement. We can assume without loss of generality that
     $m>n$. If
     $n=3$ then the statement is trivial. Indeed,
     $M(X,C_3)$ is an abelian Lie algebra. Therefore, the following formula holds in this
     algebra:
     $$\Psi= \forall g,h [g,h]=0.$$
     But
     $\Psi$ does no hold in
     $M(X, C_m)$ for
     $m\geqslant 4$. As a counterexample we can put
     $g=x_1$,
     $h=x_3$. Since
     $m>n\geqslant 3$, we have
     $m\geqslant 4$, consequently,
     $|1-3|>1$. Therefore,
     $[x_1,x_3]\neq 0$ in
     $M(X;C_n)$.

     Suppose that
     $m>n\geqslant 4$. Let
     $g_0,g_1,\dots,g_{m-1}$ make true
(\ref{longformula}) in
     $M(X;C_n)$. By Lemmas
\ref{linpartexists}--%
 \ref{linpartonecomp}, for all
     $i\in \mathbb{Z}_m$ we have
     $\widehat{g}_i=\alpha_i x_{j_i}$, where
     $\alpha_i\neq 0$.

     Easy to see that
     $|j_i-j_{i+1}|\leqslant 1$ for any
     $i\in \mathbb{Z}_m$. Indeed,
(\ref{2prod}) and homogeneity of identities and relations in a
     partially commutative metabelian Lie algebra imply
     $[\widehat{g}_i,\widehat{g}_j]=\alpha_i\alpha_{i+1}[x_{j_i},x_{j_{i+1}}]=0$, and
     the relation between
     $j_i$ and
     $j_{i+1}$ is immediate.

     Let us show that there exist
     $i,k\in \mathbb{Z}_m$ such that
     $j_i=j_k$ and therefore
     $x_{j_i}=x_{j_k}$. If there exists
     $i\in \mathbb{Z}_m$ such that
     $x_{j_i}=x_{j_{i+1}}$ then there is nothing to prove. Suppose
     that
     $|j_i-j_{i+1}|=1$ for all
     $i\in \mathbb{Z}_m$. Chose
     $i_0$, such that
     $j_{i_0}$ is a minimal among  all
     $j_i$'s. If
     $j_{i_0}\neq 0$, then
     $j_{i_0+1}=j_{i_0-1}=j_{i_0}+1$.

     Let
     $j_{i_0}=0$. If
     $j_{i_0+1}=j_{{i_0-1}}=1$ or
     $j_{i_0+1}=j_{{i_0-1}}=n-1$ then there is nothing to prove. Thus, we can suppose that
     $j_{i_0+1}\neq j_{{i_0-1}}$. Then we may assume without loss of generality that
     $j_{i_0+1}=1$,
     $j_{i_0-1}=n-1$. Since
     $j_i$ are distinct, we obtain
     \begin{equation} \label{chainineq}
       j_{i_0}<j_{i_0+1}\dots < j_{i_0+m-1}=j_{i_0-1}.
     \end{equation}
     Indeed,
     $i_0+m-1=i_0-1$ in
     $\mathbb{Z}_m$, consequently
     $j_{i+1}-j_{i}=1$ for all
     $i\in \mathbb{Z}_m$. But
     $j_{i_0-1}=m-1>n-1$. This contradiction shows that the set of
     inequalities
(\ref{chainineq}) is impossible. Therefore, for some
     $i,k\in \mathbb{Z}_m$ we have
     $j_{i}=j_{k}$. Then, by
(\ref{3prod}), we have
     $[[g_{i-1},g_{i+1}],\widehat{g}_{i}]=0$, consequently
     $[g_{i-1},g_{i+1}]\in \mathcal{C}(\widehat{g}_i)=\mathcal{C}(x_{j_i})$.
     This implies
     $[g_{i-1},g_{i+1}]\in \mathcal{C}(\widehat{g}_k)$, so
     $[[g_{i-1},g_{i+1}],\widehat{g}_{k}]=0$. But since
     $m\geqslant 5$ and
     $k\neq i$, at least one of the values
     $|k-(i-1)|$ and
     $|k-(i+1)|$ is not equal to
     $1$. This contradiction to
(\ref{longformula}) completes that proof.
   \end{proof}

 \section{To the question on algebras with the same universal theories} \label{sec4}
   Let
   $M(X;G)$ be a partially commutative metabelian Lie algebra and
   $G(X,E)$ its defining graph. For a vertex
   $x\in X$ let us put
   $x^{\perp}=\{y\,|\, d(x,y)\leqslant 1\}$, where
   $d(x,y)$ is the distance between
   $x$ and
   $y$. Namely, it is the least number
   $l$ such that there exists a path that connects
   $x$ and
   $y$ and goes through
   $l$ edges. Let us introduce a binary relation
   $\sim_{\perp}$ on
   $X$. By definition put
   $x \sim_{\perp} y$ iff
   $x_{\perp}=y_{\perp}$. Evidently, this is an equivalence relation.

   If
   $x \sim_{\perp} y$ and
   $x\neq y$ then
   $x\leftrightarrow y$. Therefore,
   $x$ and
   $y$ are in the same connected component of any subgraph of
   $G$ containing both these vertices.

   Let us prove a couple of statements about
   $\sim_{\perp}$.

   \begin{llll}\label{equivcomplgr}
     Let
     $\widetilde{X}^{\perp}_x$ be the equivalence class of
     $x$ with respect to
     $\sim_{\perp}$. Then the subgraph
     $G(\widetilde{X}^{\perp}_x)$ of
     $G$ is a complete graph. Moreover, for
     $y\in X\backslash \widetilde{X}^{\perp}_x$ we have
     $y \leftrightarrow x$ iff
     $y\leftrightarrow \widetilde{X}^{\perp}_x$.
   \end{llll}
   \begin{proof}
     Let
     $z\in \widetilde{X}^{\perp}_x$. Then by definition
     $x^{\perp}=z^{\perp}$.  Since clearly
     $x\in x^{\perp}$, we have
     $x\leftrightarrow z$.

     Since
     $\sim_{\perp}$ is an equivalence relation, we can repeat the arguments above for any
     vertex in
     $\widetilde{X}^{\perp}_x$ instead of
     $x$. So, each vertex in
     $X^{\sim_{\perp}}_x$ is adjacent to all other vertices of this set.

     Now, let
     $y\in X\backslash \widetilde{X}^{\perp}_x$. Then
     $y\leftrightarrow x$ iff
     $y \in x_{\perp}$. For any vertex in
     $z \in \widetilde{X}^{\perp}_x$ we have
     $z^{\perp}=x^{\perp}$. So, we obtain
     $y\leftrightarrow z$. Consequently, all vertices in
     $\widetilde{X}^{\perp}_x$ are adjacent to the same vertices in
     $X\backslash \widetilde{X}^{\perp}_x$. Thus,
     $y\leftrightarrow \widetilde{X}^{\perp}_x$.
   \end{proof}
   \begin{llll} \label{nonchangeconcomp}
     Let
     $G=\langle X;E\rangle$ be a graph and let
     $H=\langle Y;F\rangle$ be a subgraph of
     $G$ generated by a set
     $Y$. Let also
     $x,y\in Y$ be such that
     $x\sim_{\perp} y$ in
     $G$. Finally, let
     $H'=H(Y\backslash \{y\})$. Then the number of connected components in
     $H$ and
     $H'$ is same. Moreover, two vertices not
     equal to
     $y$ are in the same connected component of
     $H'$ iff they are in the same component of
     $H$
   \end{llll}
   \begin{proof}
     Let
     $z_1$ and
     $z_2$ be in the same connected component of
     $H$. Then there exists a path connecting
     $z_1$ and
     $z_2$:
     \begin{equation}\label{path}
       (y_0=z_1, y_1,\dots,y_{s-1},y_{s}=z_2).
     \end{equation}
     If no vertex of this path is equal to
     $y$, then there is nothing to prove. Assume the converse,
     Namely, suppose that
     $y_l=y$ for some
     $l\in\{1,2,\dots,s-1\}$. Then
     $y_{l-1}\leftrightarrow y$ and
     $y_{l+1}\leftrightarrow y$. However,
     $y\sim_{\perp} x$. Therefore, by Lemma~%
\ref{equivcomplgr},
     $y_{l-1}\leftrightarrow x$ and
     $y_{l+1}\leftrightarrow x$. So, we can change
     $y$ by
     $x$ in the path
(\ref{path}). We get another path connecting
     $z_1$ and
     $z_2$. If this path still goes through
     $y$, let us repeat the procedure described above. Finally, we can obtain a path not going through
     $y$. Consequently,
     $z_1$ and
     $z_2$ are in the same connected component of
     $H'$.

     The converse is obvious because any path in
     $H'$ is also a path in
     $H$.

     In particular, the number of connected components
     $H$ and
     $H'$ is same. This concludes the proof.
   \end{proof}

   Let
   $G$ be a graph such that
   $\sim_{\perp}$ is not identical (diagonal) on it and let
   $M(X;G)$ be a partially commutative metabelian Lie algebra with the defining graph
   $G$. Suppose that the set
   $\widetilde{X}^{\perp}_{x_{n-1}}$ contains more then one
   element. Denote by
   $X'$ the set
   $X\backslash \{x_{n-1}\}$ and by
   $G'$ the graph
   $G(X')$. Without loss of generality it can be assumed that
   $x_{n-1}\sim_{\perp} x_{n-2}$. From this point on we consider an
   order on
   $X$ for that
   $x_{n-1}$ and
   $x_{n-2}$ are two least vertices and
   $x_{n-2}<x_{n-1}$. The order of other vertices does not matter so we can fix any order having the indicated property.
   Respectively, we consider the order on
   $X'$ that is obtained from the order on
   $X$ by removing
   $x_{n-1}$.

   For any
   $\lambda \in R\backslash\{0\}$ let us define the map
   $\varphi_{\lambda}: X \to M(X',G')$ as follows:
   \begin{equation} \label{homomorph}
     \varphi_{\lambda}(x_i)=
       \begin{cases}
         x_i, & \text{if } i\neq n-1; \\
         \lambda x_{n-2} & \text{if } i=n-1.
       \end{cases}
   \end{equation}
   This map can be extended up to a homomorphism from
   $M(X;G)$ to
   $M(X';G')$ uniquely.

   Indeed, each homomorphism keeps addition and multiplication. Therefore, we can represent the image of any
   element as a linear combination of the elements that are the products of generators
   of
   $M(X;G)$. Therefore, if such homomorphism exists then it is unique.

   Let us show that extension of
   $\varphi_{\lambda}$ to the entire
   $R$-algebra Lie
   $M(X;G)$ is really a homomorphism (let us denote this extension also by
   $\varphi_{\lambda}$).  It is suffices to check that the extension keeps all identities and
   relations of the metabelian partially commutative Lie algebra
   $M(X;G)$. All identities are hold in
   $M(X';G')$ because
   $M(X';G')$ is also a metabelian Lie algebra.  Let
   $[x_i,x_j]=0$ in
   $M(X;G)$. It means that
   $x_i\leftrightarrow x_j$. If
   $i,j\neq n-1$ then
   $\varphi_{\lambda}([x_i,x_j])=[\varphi_{\lambda}(x_i),\varphi_{\lambda}(x_j)]=[x_i,x_j]=0$
   in
   $M(X';G')$. If
   $[x_i,x_{n-1}]=0$ for
   $i\neq n-1$ then
   $x_i$ is adjacent to
   $x_{n-1}$ therefore
   $x_i$ is also adjacent to
   $x_{n-2}$ and so
   $[x_i,x_{n-2}]=0$. We obtain
   $\varphi_{\lambda}([x_i,x_{n-1}])=[\varphi_{\lambda}(x_i),\varphi_{\lambda}(x_{n-1})]=\lambda[x_i,x_{n-2}]=0$.

   By
   $\mdeg([w])$ we also denote the multidegree of a homogeneous element
   $[w]$ of the algebra
   $M(X';G')$. Since it is clear what algebra we are talking about there is no ambiguity.

   Let us set
   $\lambda$ and consider
   $\varphi_{\lambda}$. Suppose that
   $[u]$ and
   $[v]$ are non-zero Lie monomials of
   $M(X;G)$ such that
   $\varphi_{\lambda}([u])$ and
   $\varphi_{\lambda}([v])$ are not equal to zero in
   $M(X';G')$. Clearly,
   $\mdeg(\varphi([u]))=\mdeg(\varphi([v]))$ iff
   $\widetilde{\mdeg}([u])=\widetilde{\mdeg}([v])$. Indeed, easy to see that
   $\mdeg(\varphi([u]))=\widetilde{\mdeg}([u])$ and
   $\mdeg(\varphi([v]))=\widetilde{\mdeg}([v])$.

   Consider an element
   $g\in M(X;G) \backslash \{0\}$.  There exists the decomposition
   \begin{equation}\label{maindecomp}
     g=\sum_{\widetilde{\delta},x_i} g_{\widetilde{\delta},x_i},
   \end{equation}
   where
   $g_{\widetilde{\delta},x_i}$ is a non-empty linear combination of basis monomials that
   appear with non-zero coefficients
   in the decomposition of
   $g$, start with
   $x_i$, and have the glued multidegree
   $\widetilde{\delta}$. It is easy to see how to obtain this decomposition. First of all, let us represent
   $g$ as a linear combination of basis monomials. Then, for each
   $g_{\widetilde{\delta},x_i}$ we need to choose the summands starting with the required generator and having the required glued multidegree.
   Since basis monomials are linearly independent, all
   $g_{\widetilde{\delta},x_i}$ that consist of at least one basis monomial with a non-zero
   coefficient are not equal to zero in
   $M(X;G)$.

   Let
   $\widetilde{\delta}_0=(\varepsilon_0,\varepsilon_1,\dots, \varepsilon_{n-2})$ be a glued multidegree.
   If representation
(\ref{maindecomp}) of
   $g$ contains
   $g_{\widetilde{\delta}_0,x_i}$ then we can write
   \begin{equation}\label{lincomb}
     g_{\widetilde{\delta}_0,x_i}=\sum_{j=0}^{\varepsilon_{n-2}} \alpha_j [u_{i,j}],
   \end{equation}
   where
   $[u_{i,j}]$ is the monomial of
   $M(X;G)$ defined as follows. Its multidegree is
   $$\mdeg([u_{i,j}])=(\varepsilon_0,\varepsilon_1,\dots,\varepsilon_{n-3},\varepsilon_{n-2}-j,j),$$
   it starts with
   $x_i$, and it is a basis monomial in the case of
   $\alpha_j \neq 0$. Let us also notice that by definition of
   $g_{\widetilde{\delta}_0,x_i}$ there exists
   $j$ such that
   $\alpha_j\neq 0$.

   \begin{llll} \label{hombas}
     If
     $[u_{i,j}]$
     is a basis monomial for some
     $j$, then
     $[u_{i,0}]$ is also a basis monomial.
   \end{llll}
   \begin{proof}
     If
     $j=0$ then there is nothing to prove.

     Let
     $j=\varepsilon_{n-2}$. Since
     $x_{n-2}\sim_{\perp} x_{n-1}$, the graphs
     $G(X_{[u_{i,\varepsilon_{n-2}}]})$ and
     $G(X_{[u_{i,0}]})$ are clearly isomorphic and the map
     taking
     $x_{n-1}$ to
     $x_{n-2}$, and
     $x_i$ to
     $x_i$ for all other
     $i$ is an isomorphism. Moreover,
     $x_{n-1}$ is the least generator, therefore
     $x_{n-1}$ is on the second place in
     $[u_{i,j}]$. Thus, the described isomorphism of graphs take
     $x_{n-1}$ to
     $x_{n-2}$ and
     $x_{n-2}$ is the least generator in
     $X_{[u_{i,0}]}$. Therefore
     $x_{n-2}$ is on the second place in
     $\varphi_{\lambda}([u_{i,j}])$. So, if
     $[u_{i,\varepsilon_{n-2}}]$ is a basis monomial then
     $x_{i}$ and
     $x_{n-1}$ are in different connected components of
     $G(X_{[u_{i,\varepsilon_{n-2}}]})$. Consequently, the images of these vertices are also in different connected components of
     $G(X_{[u_{i,0}]})$.

     Finally, suppose that
     $j\neq 0,\varepsilon_{n-2}$. By Lemma
\ref{nonchangeconcomp}, it is easy to see that
     $x_{i}$ and
     $x_{n-1}$ are in different connected components, i.e.
     $x_i$ and
     $x_{n-2}$ are also in different connected components.

     By Lemma~%
\ref{nonchangeconcomp}, for all graphs
     $G(X_{[u_{i,j}]})$ the connected components not containing
     $x_{n-2}$ and
     $x_{n-1}$ are same. Therefore if
     $x_i$ is the largest vertex in its connected component of
     $G(X_{[u_{i,j}]})$ it is also the largest vertex in the corresponding component of
     $G(X_{[u_{i,0}]})$.
   \end{proof}

   \begin{ccc}\label{corollary}
     \begin{enumerate}
       \item
         $\varphi_{\lambda}(g_{\widetilde{\delta_0},x_i})=0$ iff
         $\sum_{j=0}^{\varepsilon_{n-2}} \alpha_j \lambda^{j}=0$.
       \item
         If
         $\varphi_{\lambda}(g_{\widetilde{\delta_0},x_i})\neq 0$ then this
         is a multiple of some basis monomial.
     \end{enumerate}
    \end{ccc}
    \begin{proof}
      Note that
      \begin{equation}\label{acthomphi}
        \varphi_{\lambda}([u_{i,j}])=\lambda^{j}[u_{i,0}].
      \end{equation}
      We may write this because
      $X_{[u_{i,0}]}\subseteq X'$ and so
      $[u_{i,0}]$ can be considered as an element of
      $M(X';G')$.

      Now, one can easily obtain both assertions from Lemma~%
\ref{hombas}.
   \end{proof}

   \begin{llll}\label{subhomogeneous}
     Let
     $g$ be a non-zero element of
     $M(X;G)$ and let
     $\varphi_{\lambda}$ a homomorphism defined by
(\ref{homomorph}) for some
     $\lambda \in R\backslash\{0\}$. Then
     $\varphi_{\lambda}(g)=0$ in
     $M(X';G')$ iff
     $\varphi_{\lambda}(g_{\widetilde{\delta},x_i})=0$ for all components
     $g_{\widetilde{\delta},x_i}$ of the decomposition
(\ref{maindecomp}).
   \end{llll}
   \begin{proof}
     Let
     $g$ be a non-zero element of
     $M(X;G)$ such that
     $\varphi_{\lambda}(g)=0$. By
(\ref{maindecomp}), we obtain
     $\sum_{\widetilde{\delta},x_i} \varphi_{\lambda}(g_{\widetilde{\delta},x_i})=\varphi_{\lambda}(g)=0$.
     Suppose that
     $\varphi_{\lambda}(g_{\widetilde{\delta_0},x_{i_0}})\neq 0$
     for some glued multidegree
     $\widetilde{\delta}_0=(\varepsilon_0,\dots,\varepsilon_{n-2})$ and some
     generator
     $x_{i_0}$. Show that the following expression is not equal to
     zero in
     $M(X';G')$:
     \begin{equation}\label{homogen}
       \varphi_{\lambda} \left(\sum_{i} g_{\widetilde{\delta}_0,x_i}\right)=\sum_{i}  \varphi_{\lambda}(g_{\widetilde{\delta}_0,x_i}).
     \end{equation}

     Indeed, by
(\ref{homogen}) and Corollary~%
 \ref{corollary} we obtain
     $\sum_{i}\varphi_{\lambda} (g_{\widetilde{\delta}_0,x_i})=\sum_{i} \beta_i [u_{i,0}]$, where
     $\beta_i$ are some elements in
     $R$. Since
     $\beta_i[u_{i,0}]=\varphi_{\lambda}(g_{\widetilde{\delta}_0,x_i})$,
     some
     $\beta_i$ are not equal to
     $0$. In particular,
     $\beta_{i_0}\neq 0$. Since the first letters of the monomials
     $[u_{i,0}]$  are different they are different basis monomials. Therefore,
     $\sum_{i}\varphi_{\lambda} (g_{\widetilde{\delta}_0,x_i})\neq 0$ in
     $M(X';G')$. But this is impossible. Indeed, by homogeneity of identities and relations
     a partially commutative metabelian Lie algebra, if a Lie polynomial is equal to zero in
     $M(X';G')$ then all summands of the decomposition of this polynomial as the sum
     of homogeneous elements should also be equal to zero in
     $M(X';G')$. We are left to notice that
     $\varphi_{\lambda} \left(\sum_{i}g_{\widetilde{\delta}_0,x_i}\right)$ is just such
     summand.
   \end{proof}

   \begin{llll}\label{tononzero}
     Let
     $g \in M(X;G)\backslash \{0\}$. Then there exists
     $\lambda_0 \in \mathbb{Z}^+$ such that
     $\varphi_{\lambda}(g)\neq 0$ for any
     $\lambda\geqslant \lambda_0$.
   \end{llll}
   \begin{proof}
     It follows from Lemma~%
\ref{subhomogeneous} that
     $\varphi_{\lambda}([g])=0$ in
     $M(X';G')$ iff
     $\varphi_{\lambda}(g_{\widetilde{\delta},x_i})=0$ for all
     $g_{\widetilde{\delta},x_i}$ appearing in the decomposition
(\ref{maindecomp}). Therefore, it suffices to prove the assertion
     of the lemma in the case of
     $g=g_{\widetilde{\delta},x_i}$ that is not equal to zero in
     $M(X;G)$, where
     $x_i$ is a generator and
     $\widetilde{\delta}=(\varepsilon_{0},\varepsilon_{1},\dots, \varepsilon_{n-2})$ is a glued multidegree.

     Let
     $g=\sum_{j=0}^{\varepsilon_{n-2}}\alpha_j [u_{i,j}]$ where
     $u_{i,j}$ is a basis monomial such that it starts with
     $x_i$ and
     $\mdeg([u_{i,j}])=(\varepsilon_{0},\varepsilon_{1},\dots, \varepsilon_{n-3},\varepsilon_{n-2}-j,j)$.
     By~%
(\ref{acthomphi}), we have
     $\varphi(g)=\left(\sum_{j=0}^{\varepsilon_{n-2}} \alpha_j\lambda^j \right)[u_{i,0}]$.
     Consider the polynomial
     $p(\lambda)=\sum_{j=0}^{\varepsilon_{n-2}} \alpha_j\lambda^j$. Since
     $R$ is an integral domain, this polynomial has at most
     $\varepsilon_{n-2}$ positive integer roots. Thus,
     $\lambda_0$ can be chosen by one greater than the largest positive integer root of
     $p(\lambda)$. If
     $p(\lambda)$ has no positive integer root, we can take, for
     example,
     $\lambda_{0}=1$).
   \end{proof}

   \begin{ttt}\label{equivequiv}
     Let
     $X$ be a finite set,
     $G$ a graph. Suppose that there exists
     $x\in X$ such that the
     $\sim_{\perp}$-equivalence class
     $\widetilde{X}^{\perp}_x$ has more than one element. Finally, let
     $X'=X\backslash \{x\}$ and
     $G'=G(X')$. Then the algebras
     $M(X;G)$ and
     $M(X';G')$ are universally equivalent.
   \end{ttt}
   \begin{proof}
     By Theorem~%
\ref{univeq}, it suffices to show that for each finite submodel in
     $M(X;G)$ there exists an isomorphic submodel in
     $M(X';G')$ and vice versa.

     The converse is obviously true because
     $M(X';G')$ is a subalgebra of
     $M(X;G)$. Let us prove the direct statement.

     Let
     $x=x_{n-1 }$ and
     $x_{n-2}\sim_{\perp}x_{n-1}$. Let also
     $\Gamma=\{g_1,\dots,g_m\}$ be a finite set of the elements of
     $M(X;G)$. Extend
     $\Gamma$ adding the elements
     $g_i-g_j$,
     $g_i+g_j-g_k$,
     $[g_i,g_j]-g_k$ for all
     $i,j,k=1,2\dots ,m$ and denote by
     $\overline{\Gamma}$ the obtained set. It is sufficient to show that there
     exists
     $\lambda$ such that the kernel of
     $\varphi_{\lambda}: M(X;G)\rightarrow M(X';G')$ is disjoint with
     $\overline{\Gamma}$. If it is the case then the images of the elements in
     $\Gamma$ are distinct. Moreover, if
     $g_i\neq g_j+g_k$ or
     $g_i\neq [g_j,g_k]$ then the images of
     $g_i$ and
     $g_j+g_k$ (images of
     $g_i$ and
     $[g_j,g_k]$ respectively) are not equal either.

     By Lemma~%
\ref{tononzero}, for any non-zero
     $g\in \overline{\Gamma}$ there exists
     $\lambda_0(g)$ such that for any
     $\lambda\geqslant \lambda_0 (g)$ the following inequality
     holds:
     $\varphi_{\lambda}(g)\neq 0$.  Let
     $\lambda_0$ be maximal among
     $\lambda_0(g)$ for all
     $g\in \overline{\Gamma}$.  Then for any
     $\lambda\geqslant\lambda_0$ and for any
     $g\in \overline{\Gamma}$ we obtain
     $\varphi_{\lambda}(g)\neq 0$.

     So, the universal theories of
     $M(X;G)$ and
     $M(X';G')$ coinside.
  \end{proof}

  Let
  $G=\langle X;E\rangle$ be a graph. Suppose that there exists a
  $\sim_{\perp}$-equivalence class containing at least two
  vertices. Let
  $x$ be a vertex of such class,
  $X'=X\backslash \{x\}$, and
  $G'=G(X')$. Then by Theorem~%
\ref{equivequiv}, the universal theories of
  $M(X;G)$ and
  $M(X';G')$ coincide.
  Moreover, it is easy to see that for any equivalence relation if
  we remove an element from any equivalence class then all other elements of this class still remain in the same
  equivalence class and other equivalence classes do not change.

  So, if an obtained graph still contains a
  $\sim_{\perp}$ equivalence class with at least two vertices then we can repeat the procedure described
  above. By Theorem,~%
\ref{equivequiv} we again get a partially commutative metabelian
  Lie algebra that is universally equivalent to the initial one and so on.

  Let
  $G$ be any graph. We can remove all but one vertices from  each
  $\sim_{\perp}$-equivalence class of this graph. The universal theories of the initial and
  final graphs coincide. The obtained graph is called the
  \emph{compaction} of
  $G$. Let us denote it by
  $\overline{G}$.

  On Fig.~%
\ref{2graphs} there is an example of two graphs with the same
  universal theories such that one of them is a tree while the other one is not.
  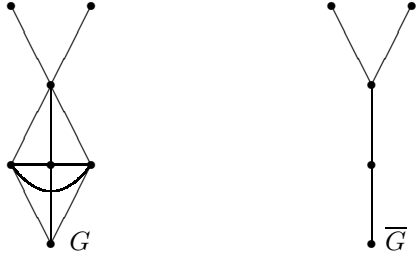
\begin{figure}
    \begin{picture}(150,100)
      \put(15,10){\line(0,1){60}}
      \put(30,40){\line(-1,2){30}}
      \put(0,40){\line(1,2){30}}
      \put(0,40){\line(1,0){30}}
      \put(15,10){\line(-1,2){15}}
      \put(15,10){\line(1,2){15}}
      \multiput(15,10)(0,30){3}{\circle*{3}}
      \multiput(0,100)(30,0){2}{\circle*{3}}
      \multiput(0,40)(30,0){2}{\circle*{3}}
      \put(135,10){\line(0,1){60}}
      \put(135,70){\line(-1,2){15}}
      \put(135,70){\line(1,2){15}}
      \multiput(135,10)(0,30){3}{\circle*{3}}
      \multiput(120,100)(30,0){2}{\circle*{3}}
      \put(22,8){$G$}
      \put(140,8){$\overline{G}$}
      \qbezier(0,40)(15,20)(30,40)
    \end{picture}
    \caption{Graph
      $G$ and it compaction
      $\overline{G}$}\label{2graphs}
  \end{figure}

  Finally, let us note that the converse of Theorem~%
\ref{equivequiv} is not true. Namely, even if the algebras
  $M(X;G)$ and
  $M(Y;H)$ are universally equivalent it does not mean that we
  can obtain
  $H$ from
  $G$ adding and removing the vertices to
  $\sim_{\perp}$-equivalence classes. Indeed, it is easy to see
  that all compactions of a graph are isomorphic. On the other hand,
  a compaction of any tree is this tree itself. But in~
\cite{PT13}, it was shown that if defining graphs of two partially
  commutative metabelian Lie algebras are trees then these algebras can be universally
  equivalent even if their defining graphs are not isomorphic.


\begin{thebibliography}{9}
     \bibitem{CF69} Cartier F., Foata D., Problems combinatorics de computation et rearrangements,
       Lecture Notes in Mathematics, 85, Springer-Verlag, Berlin, 1969.
     \bibitem{DK92}
       Duchamp G., Krob D., The lower central ceries of the free partially commutative
       group, Semigroup Forum, {\bfseries 45} (1992), 385-–394.
     \bibitem{DK92'}
       Duchamp G., Krob D., The Free Partially Commutative Lie Algebra: Bases and Ranks,
       Advances in Mathematics, {\bfseries 92}, 1992, 95--126.
     \bibitem{DK93}
       Duchamp G., Krob D., Free Partially Commutative Structures,
       Journal of Algebra, {\bfseries 156} (1993), 318--361.
     \bibitem{DKR07}
       Duncan A.\,J., Kazachkov I.\,V., Remeslennikov V.\,N.
       Parabolic and quasiparabolic subgroups of free partially commutative
       groups, Journal of Algebra, {\bfseries 318}, 2 (2007), 918--932.
     \bibitem{GT09}
       Gupta Ch.\,K., Timoshenko E.\,I., Partially Commutative Metabelian Groups: Centralizers and Elementary Equivalence,
       Algebra and Logic, {\bfseries 48}, 3 (2009), 173--192.
     \bibitem{GT11} Gupta Ch.\,K., Timoshenko E.\,I., Universal Theories for Partially Commutative Metabelian Groups,
       Algebra and Logic, {\bfseries 50}, 1 (2011), 1-16.
     \bibitem{KMNR80}
       Kim, K.\,H., Makar-Limanov, L., Neggers, J., Roush, F.\,W.,
       Graph algebras, J. Algebra, {\bfseries 64} (1980), 46--51.
     \bibitem{Por11}
       Poroshenko E.\,N., Bases for partially commutative Lie algebras,
       Algebra and Logic, {\bfseries 50}, 5 (2011),  405--417.
     \bibitem{Por12}
       Poroshenko E.\,N., Centralizers in partially commutative Lie algebras,
       Algebra and Logic, {\bfseries 51}, 4 (2012), 351--371.
     \bibitem{PT13}
       Poroshenko E.\,N., Timoshenko E.\,I., Universal Equivalence of
       Partially Commutative Metabelian Lie Algebras, J. Algebra, 2013, 143--168.
     \bibitem{Se89}
       Servatius H., Automorphisms of graph groups, J. Algebra, {\bfseries 126},
       1 (1989), 34-–60.
     \bibitem{She05}
       Shestakov S.\,L., The equation
       $[x, y] = g$ in partially commutative groups,
       Siberian Math. J., {\bfseries 46}, 2 (2005), 364–-372.
     \bibitem{She06}
       Shestakov S.\,L., The equation
       $x^2 y^2 = g$ in partially commutative groups,
       Siberian Math. J., {\bfseries 47}, 2 (2006), 383–-390.
     \bibitem{Ti10}
       Timoshenko E.\,I., Universal Equivalence of Partially Commutative Metabelian Groups,
       Algebra and Logic, {\bfseries 49}, 2 (2010), 177--196.
     \bibitem{Ti11} Timoshenko E.\,I., A Mal'tsev Basis for a Partially Commutative Nilpotent Metabelian
       Group, Algebra and Logic, {\bfseries 50}, 5 (2011), 439--446.
  \end{thebibliography}
\end{document}